\documentclass[11pt]{article}
\usepackage{epigamath}

\setpapertype{A4}

\usepackage[english]{babel}

\usepackage[dvips]{graphicx}     % Use this instead with dvips

\title{\vspace{-0.5cm}Wonderful compactifications of Bruhat-Tits buildings}
\titlemark{Wonderful compactifications of Bruhat-Tits buildings}
\author{\vspace{0cm}Bertrand R\'emy, Amaury Thuillier, and Annette Werner}
\authoraddresses{
\authordata{Bertrand R\'emy}{\firstname{Bertrand} \lastname{R\'emy}\\
\institution{CMLS, \'Ecole polytechnique, CNRS, Universit\'e
Paris-Saclay, F-91128 Palaiseau Cedex}\\
\email{bertrand.remy@polytechnique.edu}}\\
\authordata{Amaury Thuillier}{\firstname{Amaury}
\lastname{Thuillier}\\
\institution{Universit\'e de Lyon, CNRS, Universit\'e Lyon 1,
Institut Camille Jordan, 43 boulevard du 11 novembre 1918,
F-69622 Villeurbanne cedex}\\
\email{thuillier@math.univ-lyon1.fr}}\\
\authordata{Annette Werner}{\firstname{Annette}
\lastname{Werner}\\
\institution{Institut f\"ur Mathematik, Goethe-Universit\"at
Frankfurt, Robert-Mayer-Str. 6-8,
D-60325 Frankfurt a.M.}\\
\email{werner@math.uni-frankfurt.de}} }
\authormark{B. R\'emy, A. Thuillier, \&\ A. Werner}
\date{\vspace{-5ex}} % Empty date or tweak it according to your needs
\journal{\'Epijournal de G\'eom\'etrie Alg\'ebrique} % Epijournal name
\acceptation{Received by the Editors on February 10, 2017, and in final form
on October 17, 2017. \\ Accepted on November 22, 2017.}

\newcommand{\articlereference}{Volume 1 (2017), Article Nr. 10}

\usepackage[all]{xy}

\newtheorem{Thm}{Theorem}[section]

\newtheorem{Prop}[Thm]{Proposition}

\newtheorem{Lemma}[Thm]{Lemma}
\newtheorem{Remark}[Thm]{Remark}

\newcommand{\G}{\mathrm{G}}

\newcommand{\Bb}{\mathcal{B}}

\newcommand{\an}{\mathrm{an}}

\newcommand{\cT}{\mathcal{T}}
\newcommand{\bbart}{\overline{\Bb}(G,k)}

\newcommand{\Thbar}{\overline{\Theta}}

\newcommand{\Gbar}{\overline{G}^\an}

\begin{document}

%%%%%%%%%%%%%%%%%%%%%%%%%%%%%%%
% Add the title to the document
%%%%%%%%%%%%%%%%%%%%%%%%%%%%%%%

\maketitle

%\contribution{}

%%%%%%%%%%%%%%%%%%%%%
% Dedication (if any)
%%%%%%%%%%%%%%%%%%%%%
%\dedication{}

%%%%%%%%%%%%%%%%%%%%%%%%%%%%%%%%%%%%%%%%%%%%%%%%%%%%%%%%%%
% Add abstract, Keywords, MSC classification (recommended)
% Never remove prelims section, make it rather empty
%%%%%%%%%%%%%%%%%%%%%%%%%%%%%%%%%%%%%%%%%%%%%%%%%%%%%%%%%%
\begin{prelims}

%\vspace{-0.55cm}

\def\abstractname{Abstract}
\abstract{Given a split adjoint semisimple group over a local field,
we consider  the maximal Satake-Berkovich compactification of the
corresponding Euclidean building. We prove that it can be
equivariantly identified with the compactification we get by
embedding the building in the Berkovich analytic space associated to
the wonderful compactification of the group. The construction of
this embedding map is achieved over a general non-archimedean
complete ground field. The relationship between the structures at
infinity, one coming from strata of the wonderful compactification
and the other from Bruhat-Tits buildings, is also investigated.}



%\vspace{-0.05cm}

\languagesection{Fran\c{c}ais}{%

%\vspace{-0.05cm}
\textbf{Titre. Compactifications magnifiques des immeubles de
Bruhat-Tits} \commentskip \textbf{R\'esum\'e.} \'Etant donn\'e un
groupe adjoint semi-simple d\'eploy\'e sur un corps local, nous
consid\'erons la compactification de Satake-Berkovich maximale de
l'immeuble euclidien correspondant. Nous prouvons qu'elle peut
\^etre identifi\'ee de mani\`ere \'equivariante avec la
compactification obtenue en plongeant l'immeuble dans l'espace
analytique de Berkovich associ\'e \`a la compactification magnifique
du groupe. La construction de ce plongement est effectu\'ee sur un
corps complet non-archim\'edien g\'en\'eral. La relation entre les
structures \`a l'infini, l'une venant des strates de la
compactification magnifique et l'autre des immeubles de Bruhat-Tits,
est \'egalement \'etudi\'ee.}

\end{prelims}

%%%%%%%%%%%%%%%%%%%%%
% Content begins here
%%%%%%%%%%%%%%%%%%%%%

\newpage

% Add table of contents (optional)
\setcounter{tocdepth}{1}
\tableofcontents

\section*{Introduction}
\label{s - intro}

In this paper, we are interested in compactifications of algebraic
groups and of some of their related geometries. By ``related
geometries" we mean ``symmetric spaces" and this terminology can be
understood in at least two different ways. The first one is purely
algebraic and does not require any topological assumption on the
ground field: a symmetric space is then the homogeneous space given
by the quotient of an adjoint semisimple group by the identity
component of the fixed-point set of an involution; the prototype for
such a space is ${(G \times G)  / {\rm diag}(G)}$ where ${\rm
diag}(G)$ is the diagonal $\{ (g,g) : g \in G \}$. The second
meaning makes sense when the ground field is endowed with a complete
non-archimedean absolute value; then we investigate a Euclidean
building, as given by the Bruhat-Tits theory of reductive groups
over valued fields (see \cite{BruhatTits1} and \cite{BruhatTits2}).

\vskip 1.4mm

To each of the two kinds of symmetric spaces corresponds at least
one compactification procedure. The main question of this paper is
to understand, when $k$ is a non-archimedean local field, the relationship between the (so-called wonderful) projective variety
compactifying the symmetric space ${(G \times G) / {\rm diag}(G)}$
and the Satake-Berkovich compactifications of the associated
Bruhat-Tits building $\Bb(G,k)$, as previously constructed by
Berkovich in \cite{Berkovich} and  by the authors in \cite{RTW1} and \cite{RTW2}. The first space is useful for instance for the
algebraic representation theory of the group $G$ while the second
one, relevant to the analogy with the Riemannian symmetric spaces of
real Lie groups, is useful for the analytic representation theory
of, and the harmonic analysis on, the group $G(k)$.

\vskip 1.4mm

Let us be more precise and consider a split semisimple group of
adjoint type $G$ over some field $k$. Wonderful compactifications
were initially constructed by representation-theoretic methods (see
\cite{ConciniProcesi}, \cite{Strickland} and \cite{ConciniSpringer})
but can now be also constructed by using Hilbert schemes (see
\cite{BrionJAG} and \cite{BrionCMH}). We adopt the latter viewpoint
in the core of the paper, but use the former one in this
introduction for simplicity. Let $\rho: G \rightarrow \mbox{GL} (V)$
be an irreducible representation defined over $k$, assumed to have
regular highest weight (strictly speaking, one has to choose a bit
more carefully the linear representation $\rho$ in positive
characteristic -- see \cite[Lemma 1.7 and Sect.
3]{ConciniSpringer}). The projective space $\mathbb{P} \bigl( {\rm
End}(V) \bigr)$ is a $G \times G$-space for the action defined by:
$(g,g').M = gM(g')^{-1}$ for $g,g' \in G$ and $M \in {\rm End}(V)$.
Then the closure $\overline G$ of the orbit of $[{\rm id}_V]$ is the
wonderful compactification of ${(G \times G)  / {\rm diag}(G)}$.
From the very beginning, it was proved by de Concini and Procesi
that the $G \times G$-space $\overline G$ is a smooth projective
variety containing ${(G \times G) /  {\rm diag}(G)}$ as an open
orbit and with remarkable geometric properties. For instance (see
\cite{ConciniProcesi}):
\begin{itemize}
\item The boundary at infinity $\overline G \setminus G$ is a normal
  crossing divisor whose irreducible components $(D_i)_{i \in I}$ are
  indexed by the set $I$ of simple roots of the root system of
  $G$.
\item The $G \times G$-orbits are finite in number, their closures are all smooth, in one-to-one correspondence with the subsets of $I$ and there is one single closed orbit.
\item Each orbit closure fibers over the product of two flag varieties corresponding to two suitable opposite parabolic subgroups; each fiber is the wonderful compactification of the adjoint semisimple quotient of the intersection of the corresponding parabolics.
\end{itemize}

\noindent Roughly speaking, $\overline G$ does not depend on the
chosen representation and its boundary at infinity is not only nice
from the viewpoint of algebraic geometry, but also as a
Lie-theoretic object; in particular, the appearance of wonderful
compactifications of the adjoint semisimple quotients of the various
Levi factors contained in $G$ is a beautiful feature of $\overline
G$.

\vskip 2mm

We assume now that $k$ is a complete non-archimedean valued field
and we deal with the Euclidean building $\Bb(G,k)$ associated to $G$
by Bruhat-Tits theory. In \cite{RTW1}, we define a compactification
$\overline{\Bb}_\tau(G,k)$ of the building $\Bb(G,k)$ for each type
$\tau$  of parabolic subgroup, and in \cite{RTW2} we show that this
finite family of compactifications can be obtained by a suitable
analogue of Satake's compactification of Riemmanian symmetric
spaces. The compactifications we construct make crucial use of
V.~Berkovich's approach to analytic non-archimedean geometry; they
were in fact originally investigated by Berkovich in \cite[Chapter
5]{Berkovich} for split groups. This geometry allows one, and
actually requires, to use possibly huge complete non-archimedean
extensions of $k$; this explains why some of our statements are
given for arbitrary complete non-archimedean valued fields, while
for a Bruhat-Tits building $\Bb(G,k)$ to admit a compactification it
is necessary and sufficient that $k$ be a local (i.e. locally
compact) field. If $k$ is not local, then the topological space
$\overline{\Bb}_\tau(G,k)$ is not compact; however, it contains the
building $\mathcal{B}(G,k)$ as an open dense subset and the closure
of every apartment is compact. Moreover Berkovich theory associates
functorially an analytic space (with good local connectedness
properties) $X^{\rm an}$ to any algebraic $k$-variety $X$ in such a
way that if $X$ is affine, then $X^{\rm an}$ can be identified with
a suitable set of seminorms on the coordinate ring $k[X]$, and if
$X$ is proper then $X^{\rm an}$ is compact.

\vskip 2mm

In this paper, we only consider the compactification associated to
the type of Borel subgroups. It leads to the maximal
compactification among those given by the possible types, and we
denote it by $\bbart$. In \cite{RTW1}, the compactification $\bbart$
is constructed thanks to the possibility to define an embedding map
from $\Bb(G,k)$ to the Berkovich analytic space associated to the
maximal flag variety of $G$. This embedding was constructed first by
embedding the building $\Bb(G,k)$ into the Berkovich space $G^\an$,
and then by projecting to $\mathcal{F}^{\rm an}$, where
$\mathcal{F}$ is the maximal flag variety of $G$. The outcome is a
compactification whose boundary consists of the Bruhat-Tits
buildings of all the semisimple quotients of the parabolic
$k$-subgroups of $G$ \cite[Th. 4.11]{RTW1}, a striking similarity
with the algebraic case of wonderful compactifications of groups
described above.

\vskip 2mm 

In order to relate $\bbart$ to the wonderful
compactification of $G$, a natural  idea would be to use the map
$\Bb(G,k) \to G^\an$  (the first step above) and to replace the
analytification of the fibration $G \to \mathcal{F}$ (the second
step above) by the analytification of the embedding $G
\hookrightarrow \overline G, g \mapsto (g,e)$ into the wonderful
compactification. However, it turns out that the map $\vartheta :
\Bb(G,k) \rightarrow G^\an$ used for compactifying the building is
not suitable for this purpose. We have to replace it by a $G(k)
\times G(k)$-equivariant map $\Theta: \Bb(G,k)\times \Bb(G,k)
\rightarrow G^\an$ also constructed in \cite{RTW1}. This leads to
the desired comparison stated in the following theorem, which is the
main goal of this paper.

\vskip 3mm

\noindent {\bf Theorem.} {\it Let $k$ be a complete non-archimedean
field and let $G$ be a split adjoint semisimple group over $k$.
\begin{itemize}
\item[{\rm (i)}]~There exists a continuous $G(k) \times G(k)$-equivariant map $\overline{\Theta}: \Bb(G,k) \times \bbart \rightarrow \Gbar$. For every point $x$ in $\Bb(G,k)$ the map $\overline{\Theta}(x, -): \bbart \rightarrow \Gbar$ is a $G(k)$-equivariant embedding.
%The continuous $G(k)$-equivariant embedding $\Bb(G,k) \to G^\an$, where $G(k)$ acts by conjugation on $G^\an$, can be extended to a $G(k)$-equivariant embedding $\bbart \to {\overline G}^\an$ of the compactified building in the analytic wonderful compactification of $G$.
\item[{\rm (ii)}]~
When $k$ is locally compact, this embedding induces a homeomorphism
from the compactified building $\bbart$ to the closure of the image
of $\Bb(G,k) \to G^\an \to \overline G^\an$.
\item[{\rm (iii)}]~The boundaries at infinity are compatible in the
  following sense: given a proper parabolic $k$-subgroup $P$ of type
  $\tau(P)$ in $G$, the Bruhat-Tits building of the adjoint semisimple
  quotient of $P$, which is a stratum of $\bbart$, is sent into the
  analytification of the closed subscheme $\displaystyle \bigcap_{i \notin \tau(P)} D_i$.
\end{itemize}
}

\vskip 3mm

Part (i) is proven in Theorem \ref{thm:comparison}, part (ii) in
Proposition  \ref{prop - geometric.description}. Part (iii) can be
made more precise:~it is known that the intersection $\bigcap_{i
\notin \tau(P)} D_i$ is an orbit closure in the wonderful
compactification $\overline G$, and that it fibers over $G/P$ with
fibers isomorphic to the wonderful compactification of the adjoint
semisimple quotient of $P$. Then the Bruhat-Tits building at
infinity of the adjoint semisimple quotient of $P$ is sent
equivariantly to the analytification of an explicit fiber.

\vskip 2mm

The structure of this paper is as follows. Section 1 recalls the
most useful facts for us on wonderful compactifications, adopting
Brion's viewpoint using Hilbert schemes of products of a faithful
flag variety with itself. Section 2 defines the embedding maps from
Bruhat-Tits buildings to analytic spaces associated to wonderful
varieties. Section 3 investigates the boundaries of the two
compactifications; this is where part (iii) of the theorem above is
proved. Section 4 uses the results on the equivariant compatibility
of the boundaries to prove the identification between the maximal
Satake-Berkovich compactification and the one obtained thanks to
analytic wonderful varieties.

\vskip3mm {\bf Convention.} In this paper $G$ is a split adjoint
semisimple group over a field $k$. The choice of a maximal split
torus $T$ of $G$, with character group $X^*(T)$, provides a root
system $\Phi(T,G) \subset X^*(T)$. In this article, roots are always
seen as functions on $T$ and some suitable affine toric varieties
associated with $T$.

\section{Wonderful compactifications of algebraic groups}
\label{s - wonderful}

In this section, we recall the most important facts we need on
wonderful compactifications. Our main reference for this topic is
Brion's article \cite{BrionJAG}, adopting the viewpoint of Hilbert
schemes.

\vskip 2mm

Wonderful compactifications were initially (and are usually)
constructed by representation-theor\-etic means; this was first done
over an algebraically closed field of characteristic 0 by de Concini
and Procesi \cite{ConciniProcesi}, and then extended by Strickland
to the case of positive characteristic \cite{Strickland}. Brion's
paper establishes, among other things, an identification between the
wonderful compactification $\overline G$ as in the latter two papers
and an irreducible component of the Hilbert scheme ${\rm Hilb}(X
\times X)$ where $X$ is any suitable flag variety of $G$.

\vskip 2mm

Let us be more precise. Let $k$ be a field and let $G$ be a
$k$-split adjoint semisimple group. We choose a parabolic
$k$-subgroup $P$ of $G$ such that the $G$-action on the flag variety
$X=G/P$ is faithful, which amounts to requiring that $P$ does not
contain any simple factor of $G$. As before, we denote by $\overline
G$ the wonderful compactification obtained via an irreducible
representation. The variety $\overline G$ admits a $(G \times
G)$-action $(g,g',\bar g) \mapsto (g,g').\bar g$, which we denote by
$(g,g').\bar g = g \bar g (g')^{-1}$ for $g,g' \in G$ and $\bar g
\in \overline G$. This notation is motivated by the construction of
$\overline G$ itself: given a highest weight module $(V,\rho)$ (e.g.
obtained as in \cite[Lemma 1.7]{ConciniSpringer}), the
compactification $\overline G$ is the closure in $\mathbb{P} \bigl(
{\rm End}(V) \bigr)$ of the $(G \times G)$-orbit of $[{\rm id}_V]$
for the action induced by $(g,g').M = \rho(g)M\rho(g')^{-1}$ for any
$g,g' \in G$ and $M \in {\rm End}(V)$.

\vskip 2mm

We now turn specifically to Brion's approach. Let us denote by
$\overline P$ the closure of $P$ in the complete variety $\overline
G$: this space is stable under the restricted action by $P \times
P$. Let $\mathcal{G}$ be the space $(G \times G) \times^{P \times P}
\overline P$ constructed as the image of the quotient map
$$q : G \times G \times \overline P \twoheadrightarrow (G \times G) \times^{P \times P} \overline P = \mathcal{G}$$
associated to the right $(P \times P)$-action defined by:
$$(p,p').(g,g',\bar p) = (gp,g'p',p^{-1}\bar p p')$$
for all $g,g' \in G$, $p,p' \in P$ and $\bar p \in \overline P$. The
orbit of $(g,g',\bar p)$ for this action is denoted by $[g,g',\bar
p]$.

\vskip 2mm

On the one hand, the right $(P \times P)$-action on $G \times G$ is
free and the map

$$\begin{array}{r r c l}
\pi_{X \times X}: & \mathcal{G} & \to & X \times X \\
& [g,g',\bar p] & \mapsto & [g,g'] = (gP,g'P)
\end{array}$$

\noindent is a locally trivial fibration with fiber $\overline P$.
On the other hand, the $(G \times G)$-action on $\overline G$
restricted to $G \times G \times \overline P$ factors through the
quotient $q$ to give another map with the same source space as
$\pi_{X \times X}$, namely:

$$\begin{array}{r r c l}
\pi_{\overline G} : & \mathcal{G} & \to & \overline G \\
& [g,g',\bar p] & \mapsto & g \bar p (g')^{-1}.
\end{array}$$

\noindent By taking the product, one finally obtains a map

$$\begin{array}{r r c l}
\pi_{X \times X} \times \pi_{\overline G} : & \mathcal{G} & \to & X \times X \times \overline G \\
& [g,g',\bar p] & \mapsto & (gP,g'P,g \bar p (g')^{-1})
\end{array}$$

\noindent which is a closed immersion and which enables one to see
the fibers of $\pi_{\overline G}$ as a flat family of closed
subschemes of $X \times X$ (see \cite[Sect. 2, p. 610]{BrionJAG} for
more details and additional references to previous work by Brion
\cite{BrionCMH}).
% ref ?
The outcome is a $(G \times G)$-equivariant morphism obtained thanks
to the universal property of the Hilbert scheme:

$$\begin{array}{r c c l}
\varphi : & \overline G & \to & {\rm Hilb}(X \times X) \\
& \bar g & \mapsto & (\pi_{X \times X})_*\bigl((\pi_{\overline
G})^*\bar g \bigr)
\end{array}$$

\noindent which, roughly speaking, attaches to any point $\bar g$ of
the wonderful compactification $\overline G$, the following closed
subscheme of the product $X \times X$ of faithful flag varieties:

$$\varphi(\bar g) = \{ (gP,g'P) \, | \, g,g' \in G, \, \bar g \in g \overline P (g')^{-1} \} \subset X \times X.$$

\noindent This description of the images of $\varphi$ comes from the
whole description of the image $(\pi_{X \times X} \times
\pi_{\overline G})(\mathcal{G})$ as an ``explicit" incidence variety
in $X \times X \times \overline G$ [loc. cit.]. It provides an easy
way to compute that $\varphi(1_G)$ is the diagonal subscheme ${\rm
diag}(X)$ in $X \times X$, a point in ${\rm Hilb}(X \times X)$ whose
stabilizer for the induced $(G \times G)$-action is easily seen to
be the diagonal subgroup ${\rm diag}(G)$ of $G \times G$. Therefore,
using the latter facts together with the $(G \times G)$-equivariance
of $\varphi$, one can see $\varphi(\overline G)$ as the space of
degeneracies of the diagonal ${\rm diag}(X)$ in $X \times X$, the
images of the elements of $\displaystyle G = {(G \times G)  / {\rm
diag}(G)}$ being the graphs of the elements $g$ seen as
automorphisms of $X=G/P$. We will use a more detailed understanding
of the boundary points in section \ref{s - boundary}, but can
already quote Brion's comparison theorem \cite[Theorem 3]{BrionJAG}:

\begin{Thm}
\label{th - Brion} Let $\mathcal{H}_{X,G}$ denote the closure of the
$(G \times G)$-orbit of ${\rm diag}(X)$ in ${\rm Hilb}(X \times X)$
endowed with its reduced subscheme structure. Then the map $\varphi$
establishes a $G \times G$-equivariant isomorphism between the
wonderful compactification $\overline G$ and $\mathcal{H}_{X,G}$.
\end{Thm}

\noindent Note that when $G = {\rm Aut}(X)^\circ$ (which is the case
in general \cite{Demazure}), the space $\mathcal{H}_{X,G}$ is also
the irreducible component of ${\rm Hilb}(X \times X)$ passing
through ${\rm diag}(X)$ \cite[Lemma 2]{BrionJAG}. Note also that the
above isomorphism holds for any parabolic $k$-subgroup provided the
associated flag variety is a faithful $G$-space. The Lie-theoretic
construction of the wonderful compactification of $G$ in
\cite{Strickland}, written over an algebraically closed base field
of arbitrary characteristic, applies more generally when $G$ is
\emph{split} over an arbitrary field. In \cite{BrionJAG}, Brion
always works over an algebraically closed field. However, his
construction makes sense over an arbitrary field, and its naturality
allows one to see that, by faithfully flat descent, the statement of
the above theorem remains true for a \emph{split} semisimple group
$G$ of adjoint type over an arbitrary field $k$.

%any an easy Galois descent can be performed from a ground field equal to the algebraic closure $\overline k$ of $k$ (strickly speaking, the context taken care of in the paper \cite{BrionJAG}) to the  field $k$ itself.
% shall we argue more?

\vskip 2mm

The proof of this theorem uses a lot of knowledge on the structure
of $\overline G$, previously obtained by representation theory (see
\cite{ConciniProcesi}, \cite{Strickland} and
\cite{ConciniSpringer}). In fact, once some standard Lie-theoretic
choices have been made in $G$, the latter considerations exhibit in
the wonderful compactification $\overline G$, as a main tool of
study of it, an explicit open affine subset $\overline G_0 \subset
\overline G$. More precisely, let $T$ be a split maximal $k$-torus
in $G$ with character group $X^*(T)$, let $B_+$ and $B_-$ be two
opposite Borel subgroups containing $T$ and with unipotent radical
$U_+$ and $U_-$, respectively. These choices provide as usual a root
system $\Phi = \Phi(T,G) \subset X^*(T)$ and two opposite subsets
$\Phi_+$ and $\Phi_-$ corresponding to the roots appearing in the
adjoint $T$-action on the Lie algebras of $U_+$ and $U_-$,
respectively. The affine open subset $\overline G_0$ satisfies the
following properties (see for instance \cite[Section 2]{Strickland}
for the original reference in arbitrary characteristic):

\begin{itemize}
\item the subset $\overline G_0$ contains $T$ and is stable under the action by $TU_- \times TU_+$;
\item the closure of $T$ in $\overline G_0$, which we denote by $Z$, is the affine toric variety associated with the semigroup $\langle \Phi^- \rangle$ of $X^*(T)$ spanned by the negative roots;
\item the canonical map $U_- \times Z \times U_+ \rightarrow \overline G$ is an isomorphism onto $\overline G_0$;
\item the subset $Z$ is isomorphic to an affine space of dimension equal to ${\rm dim}(T)$, therefore $\overline G_0$ is isomorphic to an affine space of dimension equal to ${\rm dim}(G)$ since $U_+ \simeq U_- \simeq {\bf A}_k^d$, with $2d = |\Phi|$, as $G$ is split;
\item the $(G \times G)$-orbits in $\overline G$ are in one-to-one correspondence with the $(T \times T)$-orbits in the toric affine variety $Z$.
\end{itemize}

\noindent In what follows, we see the affine space $Z$ as a partial
compactification of the split torus $T \simeq (\mathbb{G}_m)^{{\rm
dim}(T)}$. Moreover there is a simple way to construct a complete
set of representatives of the $(T \times T)$-orbits in $Z$ by 
``pushing to infinity" the diagonal ${\rm diag}(X)$ by suitable
one-parameter subgroups in $T$. For instance, given any regular
one-parameter subgroup $\lambda : \mathbb{G}_m \to T$, the limit
$$\displaystyle \lim_{t \to 0} \, (\lambda(t),1).{\rm diag}(X)$$
exists and is, so to speak, the ``most degenerate degeneracy" of the
diagonal; it is also the point of $Z$ in the unique closed $(G
\times G)$-orbit of $\overline G$.

\vskip 2mm

Roughly speaking, the next section, where our embedding map is
defined, is the Berkovich analytic  counterpart of some of the
previous facts.

\section{Construction of the embedding map}
\label{s - embedding}

We henceforth assume that the field $k$ is complete with respect to
a non-trivial non-archimedean absolute value, and we keep the
adjoint split semisimple $k$-group $G$ as before. A non-archimedean
field extension of $k$ is a field $K$ containing $k$, which is
complete with respect to a non-archimedean absolute value extending
the one on $k$.

Our main goal in this section is to construct an equivariant map
from the Bruhat-Tits building $\Bb(G,k)$ to the Berkovich analytic
space $\overline G^{\rm an}$ associated to the wonderful
compactification $\overline{G}$ of $G$. Note that $\overline G^{\rm
an}$ is compact since $\overline{G}$ is proper. More precisely, we
define a  continuous equivariant map from $\Bb(G,k) \times
\overline{\Bb}(G,k) $ to $\overline G^{\rm an}$. Fixing a special
point in the first coordinate gives the map we aim for. Later, in
section \ref{s - equivariant comparison}, we will show that this map
is an embedding.

\vskip 2mm

Let us first recall some important facts  on Satake-Berkovich
compactifications of buildings (see \cite{RTW1} and \cite{RTW2} for
details). In \cite[Prop. 2.4]{RTW1} we define a morphism $\vartheta:
\Bb(G,k) \rightarrow G^\an$ by associating to each point $x$ in the
building $\Bb(G,k)$ a $k$-affinoid subgroup $G_x$ of $G^\an$ (the
underlying set of $G_x$ is an affinoid domain of $G^{\rm
an}$). The subgroup $G_x$ is an analytic refinement of the integral
structure of $G$ associated to $x$ by Bruhat-Tits theory \cite[4.6
and 5.1.30]{BruhatTits2}. Working in an analytic context, rather
that in a purely algebraic one, has the important advantage that two
distinct points, even in the same facet (i.e. the same cell) of
$\Bb(G,k)$, lead to distinct analytic subgroups. We use the group
$G_x$ to  define the image $\vartheta(x)$ as the unique Shilov
boundary point of $G_x$ \cite[2.4]{Berkovich}. The map $\vartheta$
obtained in this way is a continuous $G(k)$-equivariant injection if
we let $G(k)$ act on $G^\an$ by conjugation \cite[Prop. 2.7]{RTW1}.
The map $\vartheta$ is well-adapted to Bruhat-Tits theory in the
sense that for any non-archimedean extension $K/k$, the group
$G_x(K)$ is the stabilizer in $G(K)$ of $x$, seen as a point in the
building $\Bb(G,K)$.

\vskip 2mm

Unfortunately, we cannot use the natural map $\vartheta$ to define
an embedding towards $\overline G^{\rm an}$ that could be useful for
our purposes. We have to use another one, also constructed [loc.
cit.] thanks to $\vartheta$. It is a continuous morphism
\[\Theta: \Bb(G,k) \times \Bb(G,k) \rightarrow G^\an\]
which can be seen as a map describing a kind of ``relative position"
from one point to another. This viewpoint gives an intuition to
understand why the equivariance relation
\[\Theta(gx, hy) = h \Theta(x,y) g^{-1}\]
is satisfied by $\Theta$ for  all $x, y \in \Bb(G,k)$ and $g,h \in
G(k)$ \cite[Prop. 2.11]{RTW1}. The definition of $\Theta$ is again
an improvement of facts known from Bruhat-Tits theory -- here, the
transitivity properties of the $G(k)$-action on the facets of
$\Bb(G,k)$ -- made possible by the Berkovich analytic viewpoint.
Indeed, this viewpoint is flexible enough to allow the use of
(possibly huge) non-archimedean extensions of $k$ in order to obtain
better transitivity properties. More precisely, for $x, y \in
\Bb(G,k)$ there exists an extension $K/k$ as before and an element
$g \in G(K)$ such that after embedding $\Bb(G,k) $ into $\Bb(G,K)$
we have $g x = y$. Then we define $\Theta(x,y)$ to be the image of
$g \vartheta_K(x)$ under the natural projection from $G_K^\an$ to
$G^\an$, where $G_K$ is the base change of $G$ by $K$, and
$\vartheta_K: \Bb(G,K) \rightarrow G_K^\an$ is the above embedding
over $K$. Note that $\Theta$ is compatible with non-archimedean
field extensions  and that, if $G$ is reductive, we can define the
map $\Theta$ on the extended building of $G$ (which then contains
the building of the semisimple group $[G,G]$ as a factor). Moreover,
by the same Proposition we know that for every point $x_0$ in
$\Bb(G,k)$ the map $\Theta(x_0,-): \Bb(G,k) \rightarrow G^{\an}$ is
a $G(k)$-equivariant injection, where $G(k)$ acts by left
translations on $G^\an$.

\vskip 2mm

The key result for our comparison theorem in section 4 is the
following statement. It gives a map which, when the first argument
is fixed, is eventually shown to be the embedding we are looking
for.

\begin{Prop}
\label{prop - embedding} The map $\Theta: \Bb(G,k) \times \Bb(G,k)
\rightarrow G^\an$ has a continuous extension
\[\Thbar: \Bb(G,k) \times \overline{\Bb}(G,k) \rightarrow \overline G^\an,\]
such that for all $g,h \in G(k)$, $x \in \Bb(G,k)$ and $y \in
\overline \Bb(G,k)$, we have
$$\Thbar(gx, hy) = h \Thbar(x,y) g^{-1}.$$
The map $\Thbar$ is compatible with non-archimedean field
extensions: if $k'/k$ is a non-archimedean extension, then the
natural diagram
\[
\xymatrix{\mathcal{B}(\G,k') \times \overline{\mathcal{B}}(G,k')
\ar@{->}[r]^{\hspace{1cm} \Thbar} &
(\overline{G} \otimes_k k')^{{\rm an}} \ar@{->}[d]^{{\rm pr}_{k'/k}} \\
\mathcal{B}(G,k) \times \overline{\mathcal{B}}(G,k) \ar@{->}[u]
\ar@{->}[r]_{\hspace{1cm} \Thbar} & \overline{G}^{{\rm an}}}\] is
commutative.
\end{Prop}

The rest of the section is dedicated to the proof of this statement.
We start with some auxiliary results.

\vskip 2mm

Consider a maximal split torus $T$ of $G$ and a Borel subgroup $B$
of $G$ containing $T$: such an inclusion $T \subset B$ will
henceforth be called a {\it standardization} in $G$. We fix a
standardization $(T,B)$ in $G$. The Borel group $B$ gives rise to an
order on the root system $\Phi = \Phi(T,G)$ inside the character
group $X^*(T)$ of $T$, and we denote the corresponding set of
positive (resp. negative) roots by $\Phi^+$ (resp. $\Phi^-$).
Moreover, we denote the associated unipotent subgroups by $U_+ =
\prod_{\alpha \in \Phi^+} U_\alpha$ and $U_- = \prod_{\alpha \in
\Phi^-}U_\alpha$; they are the unipotent radicals of $B$ and of its
opposite with respect to $T$, respectively.

\vskip 2mm

We denote by $A$ the apartment associated to $T$ in $\Bb(G,k)$: by
Bruhat-Tits theory, it is an affine space under the real vector
space $X_*(T) \otimes_{\bf Z} {\bf R}$ where $X_*(T)$ is the
cocharacter group of $T$, but in accordance with \cite{RTW1} we will
see it as an affine space under $\Lambda = {\rm Hom}_{\bf
Ab}(X^*(T),{\bf R}_{>0})$, using the multiplicative convention for
the sake of compatibility with later seminorm considerations. Now,
we pick a special point $x_0$ in $A$ and we consider the associated
{\it \'epinglage} \cite[3.2.1-3.2.2]{BruhatTits2}: this is a
consistent choice of coordinates $\xi_\alpha : U_\alpha \,
\tilde{\rightarrow} \, {\mathbb G}_{a,k}$ for each root $\alpha$,
which identifies the filtration of the root group $U_\alpha$ with
the canonical filtration of ${\mathbb G}_{a,k}$. Thus we get an
isomorphism between the big cell $U_- \times T \times U_+$ and the
spectrum of the $k$-algebra $k[X^*(T)][(\xi_\alpha)_{\alpha \in
\Phi}]$. We also use $x_0$ to identify the apartment $A$ with
$\Lambda$: thus there is a natural pairing $\langle \; , \; \rangle$
between $A$ and $X^\ast (T)$, which we can restrict to a pairing
between $A$ and $\Phi$.

\vskip 2mm

At last, we recall that the underlying space of the analytic space
associated to an affine $k$-variety $V$ is the set of multiplicative
seminorms $k[V] \to {\bf R}$, defined on the coordinate ring of $V$
and extending the absolute value of $k$. Therefore a point in the
analytic big cell $(U_- \times T \times U_+)^{\rm an}$ is a
multiplicative seminorm on the $k$-algebra
$k[X^*(T)][(\xi_\alpha)_{\alpha \in \Phi}]$.

\vskip 2mm

We first show an explicit formula for the restriction of $\Theta$ to
$A \times A$.

\begin{Lemma}
\label{lemma - explicit formula seminorms} We use the notation
introduced above. For each $(x,y) \in A \times A$, the point
$\Theta(x,y)$, a priori in $G^{\rm an}$, actually lies in $(U_-
\times T \times U_+)^\an$. It is given by the following
multiplicative seminorm on the coordinate ring
$k[X^*(T)][(\xi_\alpha)_{\alpha \in \Phi}]$ of $U_- \times T \times
U_+$:
\[\sum_{\chi \in X^*(T), \nu \in {\bf N}^\Phi} a_{\chi, \nu} \chi \xi^\nu
\mapsto \max_{\chi,\nu} |a_{\chi,\nu}| \langle y,\chi \rangle
\langle x,\chi\rangle^{-1} \prod_{\alpha \in \Phi_-} \langle
y,\alpha\rangle^{\nu(\alpha)} \prod_{\alpha \in \Phi_+} \langle
x,\alpha\rangle^{\nu(\alpha)}.\]
%\noindent
\end{Lemma}

Note that the seminorm $\Theta(x,y)$ is in fact a norm.

\begin{proof}
To check this formula, we first observe that $\Theta(x_0, x_0)
= \vartheta(x_0)$, so that the desired formula for $\Theta(x_0,x_0)$
follows from \cite[Prop. 2.6]{RTW1}. Given $(x,y) \in A \times A$,
there exist a non-archimedean field extension $K/k$ and points $s,t
\in T(K)$ such that we have $\langle x, \chi\rangle = |\chi(t)|$ and
$\langle y, \chi\rangle = |\chi(s)|$ for any $\chi \in X^*(T)$. By
compatibility of $\Theta$ with non-archimedean field extensions and
$G(K) \times G(K)$-equivariance \cite[Prop. 2.11]{RTW1}, we can
write $\Theta(x,y) = s\Theta(x_0,x_0)t^{-1}$. Since
\[ \chi(s w  t^{-1})
= \chi(s)\chi(t)^{-1} \chi(w)\ \ {\rm and } \quad \xi_\alpha(s v
s^{-1}) = \alpha(s)\xi_{\alpha}(v) \]

%$\chi(s w  t^{-1})
%= \chi(s)\chi(t)^{-1} \chi(w)\ \ {\rm and} \ \ \xi_\alpha(s v  t^{-1})
%= \left\{\begin{array}{ll} \alpha(s)\xi_{\alpha}(v) & {\rm if } \ \alpha \in \Phi_- \\ \alpha(t)\xi_{\alpha}(v)  & {\rm if} \ \alpha \in \Phi_+\end{array}\right.$$
\noindent we deduce, for $f = \sum_{\chi,\nu} a_{\chi,\nu} \chi
\xi^\nu \in k[X^*(T)][(\xi_\alpha)_{\alpha \in \Phi}]$, that

\begin{eqnarray*}
  |f|(\Theta(x,y)) & = & |f|(s\Theta(x_0,x_0)t^{-1}) \\
  & = &
  \left|\sum_{\chi,\nu} a_{\chi,\nu} \chi(s)\chi(t)^{-1} \prod_{\alpha
    \in \Phi_-}\alpha(s)^{\nu(\alpha)} \prod_{\alpha
    \in \Phi_+} \alpha(t)^{\nu(\alpha)} \chi \xi^\nu
  \right|(\Theta(x_0,x_0)) \\
  & = & \max_{\chi, \nu}
  |a_{\chi,\nu}|\langle y,\chi\rangle \langle x,\chi\rangle^{-1}
  \prod_{\alpha \in \Phi_-}\langle y,\alpha \rangle^{\nu(\alpha)}
  \prod_{\alpha \in \Phi_+} \langle x,\alpha \rangle^{\nu(\alpha)}.
\end{eqnarray*}
This finishes the proof. \hfill $\Box$ \end{proof}

Let us define a partial compactification of the vector space
$\Lambda = {\rm Hom}_{\bf Ab}(X^*(T),{\bf R}_{>0})$ by embedding it
into
$$\overline{\Lambda}^B = {\rm Hom}_{\bf Mon}(\langle \Phi^- \rangle, {\bf R}_{\geqslant 0}),$$
where $\langle \Phi^-\rangle$ denotes the semigroup spanned by
$\Phi^- = - \Phi^+$ in $X^*(T)$. The affine space $A$ directed by
$\Lambda$ admits a canonical $\Lambda$-equivariant compactification
$\overline{A}^B$ which can be defined as a contracted product :
$$\overline{A}^B = \overline{\Lambda}^B \times^{\Lambda} A = (\overline{\Lambda}^B \times A) / {\rm diag}(\Lambda).$$
%Since every translation of $\Lambda$ extends to a homeomorphism of $\overline{\Lambda}^B$, this defines also a partial compactification $\overline{A}^B$ of $A$.
%  of $A$ such that, for any point $x_0$ in $A$, the map $\overline{\Lambda}^B \rightarrow \overline{A}^B, \ v \mapsto x_0+v$ is a homeomorphism.
The next step is now, for each standardization $(T,B)$, to use the
previous formula in order to extend $\Theta|_{A \times A}$ to a
continuous map $\overline{\Theta}_{(T,B)} : A \times \overline{A}^B
\rightarrow \Gbar$.

\vskip 2mm

For this, we need to provide additional details about wonderful
compactifications, in particular about the affine charts given by
partially compactifying the maximal torus, seen as a factor of the
big cell. More precisely, we use the affine subvariety $\overline
G_0 \simeq U_- \times Z \times U_+$ of $\overline G$ introduced in
section~\ref{s - wonderful}. The difference between the latter
variety and the big cell is that the factor $T \simeq ({\mathbb
G}_m)^{{\rm dim}(T)}$ is replaced by a partial compactification $Z$
which is an affine space of dimension equal to ${\rm dim}(T)$. At
the level of coordinate rings, it means replacing the $k$-algebra
$k[X^*(T)][(\xi_\alpha)_{\alpha \in \Phi}]$ of the big cell, by the
$k$-algebra $k[\langle \Phi^-\rangle][(\xi_\alpha)_{\alpha \in
\Phi}]$ of $\overline G_0$.
% At the level of analytic spaces, we can see the apartment $A$ as a subset of $T^{\rm an}$ while now $\overline{A}^B$ can be seen as a subset of $Z^{\rm an}$.
% {\color{blue} Do we want to make this more precise?}

\begin{Prop}
\label{prop - extension.apartment} Fix a standardization $T \subset
B$ of $G$ with associated apartment $A$ and partial compactification
$\overline{A}^B$. Then the restriction $\Theta|_{A \times A}$
extends to a continuous embedding
$$\overline{\Theta}_{(T,B)} : A \times \overline{A}^B \rightarrow \Gbar.$$
The map $\overline{\Theta}_{(T,B)}$ actually takes its values in
$(\overline G_0)^{\rm an}$.
\end{Prop}

\begin{proof}
We use again $x_0$ to identify $A$ with ${\rm Hom}_{\bf
Ab}(X^*(T),{\bf R}_{>0})$ and $\overline{A}^B$ with ${\rm Hom}_{\bf
Mon}(\langle \Phi^-\rangle, {\bf R}_{\geqslant 0})$. Thanks to the
formula for the restriction of $\Theta$ to $A \times A$ proven in
Lemma \ref{lemma - explicit formula seminorms}, we can easily extend
this map to a continuous map $\overline{\Theta}_{(T,B)} : A \times
\overline{A}^B \rightarrow \Gbar$ by mapping $(x,y) \in A \times
\overline{A}^B$ to the multiplicative seminorm on the coordinate
ring $k[\langle \Phi^-\rangle][(\xi_\alpha)_{\alpha \in \Phi}]$ of
$\overline G_0$ defined by
$$\sum_{\chi \in \langle \Phi^-\rangle, \nu \in {\bf N}^\Phi} a_{\chi,\nu} \chi \xi^\nu \mapsto \max_{\chi,\nu} |a_{\chi,\nu}| \langle y, \chi \rangle \langle x,\chi\rangle^{-1}
\prod_{\alpha \in \Phi_-} \langle y,\alpha\rangle^{\nu(\alpha)}
\prod_{\alpha \in \Phi_+} \langle x,\alpha \rangle^{\nu(\alpha)}.$$
The right hand side is obviously continuous in $x$ and $y$, hence
$\overline\Theta$ is continuous.

\vskip 2mm

Moreover we have
$$\langle x,\alpha \rangle = |\xi_\alpha|(\overline{\Theta}_{(T,B)}(x,y))^{-1}$$
for each root $\alpha \in \Phi^+$ and
$$\langle y, \alpha \rangle = |\xi_{\alpha}|(\overline{\Theta}_{(T,B)}(x,y))$$
for each root $\alpha \in \Phi^-$. Since $\Phi^+$ spans $X^*(T)$, we
thus can recover $x$ and $y$ from $\overline{\Theta}_{(T,B)}(x,y)$
and therefore $\overline{\Theta}_{(T,B)}$ is injective.
%Let us finally check that the image of $\overline{\Theta}_{(T,B)}$ is
%closed. Pick a sequence $(x_n,y_n)$ in $A \times \overline{A}^B$ such
%that the sequence $(z_n) = (\overline{\Theta}_{(T,B)}(x_n,y_n))$
%converges in $X^{\rm an}$. Then $\langle y_n,\alpha\rangle =
%|\xi_\alpha|(z_n)$ converges in ${\bf R}_{\geqslant 0}$ for any root $\alpha \in \Phi^-$ [...]
\hfill $\Box$ \end{proof}

\begin{Remark} \label{rmk-standardization}
{\rm We recall here that, according to \cite[Prop. 4.20, (i)]{RTW1},
given any pair $(x,y) \in \Bb(G,k) \times \overline{\Bb}(G,k)$,
there exists a standardization $(T, B)$ such that $(x,y) \in A
\times \overline{A}^B$ for the apartment $A$ given by $T$. In other
words, any $(x,y) \in \Bb(G,k) \times \overline{\Bb}(G,k)$ lies in
the domain of at least one map $\Theta_{(T,B)}$.}
\end{Remark}

Our next task is to verify that the extensions $\Theta_{(T,B)}$ glue
nicely together when the standardization $(T,B)$ varies, in order to
be able to define the map $\overline \Theta$ we seek for.

\vskip 2mm

Let us recall some facts about the compactification $\bbart$. First,
we explain the relationship between the closure of an apartment $A$
in the maximal Satake-Berkovich compactification
$\overline{\Bb}(G,k)$ and the partial compactification
$\overline{A}^B$ we used so far in this section. We introduce the
maximal flag variety  $\mathcal{F} = G/B $ of $G$ (where $B$ is some
Borel subgroup of $G$), and we let  $\lambda: G \rightarrow
\mathcal{F}$ be the corresponding projection. Then the map
$\vartheta_\varnothing = \lambda^\an \circ \vartheta: \Bb(G,k)
\rightarrow \mathcal{F}^\an$ is a $G(k)$-equivariant injection
\cite[Prop. 3.29]{RTW1}. Let $A$ be an apartment in $G$ associated
to the split torus $T$. We denote by $\overline{A}$ the closure of
$\vartheta_\varnothing(A)$ in $\mathcal{F}^\an$: this is a compact
topological space. By \cite[Prop. 3.35]{RTW1}, the subset
$\overline{A}$ is homeomorphic to the compactification of $A$ with
respect to the Weyl fan, i.e. the fan consisting of the cones
\[\mathfrak{C}(P) = \{x \in A: \alpha(x) \leqslant 1 \mbox{ for all } \alpha \in -\Phi(T, P)\} , \]
where $P$ runs over all parabolic subgroups in $G$ containing $T$.
The partial compactification $\overline{A}^B$ of the present paper
is a subset of $\overline{A}$, where only the cone $\mathfrak{C}(B)$
is compactified.

\vskip 2mm

The space $\bbart$ is defined as the image of the map
\[G(k) \times \overline{A} \rightarrow \mathcal{F}^\an, \quad (g,x) \mapsto g x g^{-1}\]
endowed with the quotient topology. If the field $k$ is locally
compact, then $\bbart$ is the closure of the image of $\Bb(G,k)$ in
$\mathcal{F}^\an$ via $\vartheta_\varnothing$ and hence compact
\cite[Prop. 3.34]{RTW1}. At last, the space $\bbart$ is the disjoint
union of all $\Bb(P_{ss},k)$, where $P$ runs over all parabolic
subgroups of $G$, and where $P_{ss}$ denotes the semisimplification
$P / R(P)$ of $P$ \cite[Th. 4.1]{RTW1}.

\begin{Lemma}
\label{lemma.sequences} Let $x$ be a point in
$\overline{\mathcal{B}}(G,k)$. For any two apartments $A$ and $A'$
of $\mathcal{B}(G,k)$ whose closure in $\bbart$ contains $x$, there
exists a sequence of points in $A \cap A'$ which converges to $x$.
\end{Lemma}

\begin{proof}
% We choose a special point $x_0$ in $A$, which we use to identify the apartment $A$  and its partial compactification $\overline{A}^B$, respectively,  with the vector space ${\rm Hom}_{\bf Ab}(X^*(T),{\bf R}_{>0})$ and  its partial compactification  ${\rm Hom}_{\bf Mon}(\langle \Phi^-\rangle, {\bf R}_{\geqslant 0})$, respectively.
% We also use $x_0$ to fix coordinates $\xi_{\alpha} : U_\alpha \tilde{\rightarrow} {\mathbb G}_{a,k}$ identifying the valuation of the root datum with the canonical filtration on ${\mathbb G}_{a,k}$.
The stabilizer $G_x(k)$ of $x$ in $G(k)$ acts transitively on the
set of compactified apartments containing $x$  \cite[Prop. 4.20
(ii)]{RTW1}, hence we can write $A' = g.A$ with $g \in G_x(k)$. Pick
a standardization $(T,B)$ such that $A = A(T)$ and $x$ belongs to
$\overline{A}^B$ (cf. Remark \ref{rmk-standardization}). The
assertion is trivially true if $x$ belongs to $\mathcal{B}(G,k)$,
hence we may assume that $x$ lies at the boundary of $A$. Then there
exists a proper parabolic subgroup $P$ of $G$ containing $B$ such
that $x$ lies in the boundary stratum $\Bb(P_{ss},k)$ of $\bbart$.

\vskip 2mm

We let $N$ denote the normalizer of $T$ in $G$ and recall that $\Phi
= \Phi(T,G)$. By \cite[Th. 4.14]{RTW1}, the group $G_x(k)$ is
generated by the stabilizer $N(k)_x$ of $x$ in $N(k)$, the full root
groups $U_\alpha(k)$ when the root $\alpha$ belongs to
$\Phi(T,R_u(P))$, and the partial root groups $U_\alpha(k)_{-\log
\alpha(x)}$ for $\alpha \in \Phi(T,L)$, where $L$ is the Levi
subgroup of $P$ containing $Z_G(T) = T$. The group $N(k)$ acts on
$A$ by reflections through root hyperplanes, i.e. affine hyperplanes
parallel to a linear hyperplane of the form $\{u \ | \ \langle
u,\alpha \rangle = 1\}$ in the vector space $\Lambda = {\rm
Hom}(X^*(T),{\bf R}_{>0})$, with $\alpha \in \Phi$. Identifying $A$
and $\Lambda$, it follows that the group $N(k)_x$ fixes each point
of the closure $\overline{A}^B_x$ of the affine subspace
$$A_x = \{y \in A \ | \alpha(y) = \alpha(x) \mbox{ for all } \alpha \in \Phi^- \mbox{such that} \ \alpha(x)\neq 0\}.$$
Consider a root $\alpha \in \Phi$ and an element $u$ in
$U_\alpha(k)$. The action of $u$ on $\mathcal{B}(G,k)$ fixes each
point of the half-space
$$A_u = \{y \in A \ | \ \alpha(y) \geqslant |\xi_\alpha(u)|\}.$$
The closure of the latter in $\overline{A}^B$ is the subspace
$\overline{A}^B_u$ defined by
$$\overline{A}^B_u = \{y \in \overline{A}^B \ | \ \alpha(y) \geqslant |\xi_\alpha(u)|\}  \mbox{  if } \alpha \in \Phi^-$$
and
$$\overline{A}^B_u = \{y \in \overline{A}^B \ \ | \ \ (- \alpha)(y)|\xi_\alpha(u)| \leqslant 1\} \mbox{ if } \alpha \in \Phi^+.$$
In each case, if $\overline{A}^B_u$ contains $x$, then
$\overline{A}_u^B \cap \overline{A}^B_x$ is a neighborhood of $x$ in
$\overline{A}^B_x$.

\vskip 2mm

Now, thanks to the description of $G_x(k)$ recalled above, any given
element $g$ of $G_x(k)$ fixes each point in the intersection of
$\overline{A}^B_x$ with a finite number of (partially) compactified
half-spaces $\overline{A}^B_u$ all containing $x$, hence fixes each
point in some neighborhood $V$ of $x$ in $\overline{A}^B_x$. We
deduce $V \subset \overline{A}^B \cap g\overline{A}^B$ and,
therefore, there exists a sequence of points in $A \cap gA$ which
converges to $x$ in both $\overline{A}^B$ and
$g\overline{A}^B$.~\hfill $\Box$ \end{proof}

\begin{proof}[Proof of Proposition \ref{prop - embedding}.]
We are now
in position to prove successively the properties claimed about the
map $\overline \Theta$.

\vskip 2mm

1) Existence. We first check that the maps
$\overline{\Theta}_{(T,B)}$ glue together nicely. Pick two points $x
\in \Bb(G,k)$ and $y \in \overline{\Bb}(G,k)$. We have to check that
$$\overline{\Theta}_{(T,B)}(x,y) =
\overline{\Theta}_{(T',B')}(x,y)$$ for any two standardizations
$(T,B)$ and $(T',B')$ such that $\overline{A(T)}^B$ and
$\overline{A(T')}^{B'}$ both contain $x$ and $y$. By Lemma
\ref{lemma.sequences}, we can pick a sequence $(y_n)_{n \geqslant
0}$ in $A(T) \cap A(T')$ converging to $y$ in $\overline{A(T)}^B$
and $\overline{A(T')}^{B'}$. We have
$$\overline{\Theta}_{(T,B)}(x,y_n) = \Theta(x,y_n) =
\overline{\Theta}_{(T',B')}(x,y_n)$$ for all $n$, hence
$\overline{\Theta}_{(T,B)}(x,y) = \overline{\Theta}_{(T',B')}(x,y)$
by continuity of $\overline{\Theta}_{(T,B)}$ and
$\overline{\Theta}_{(T',B')}$.

Since any two points $x \in \Bb(G,k)$ and $y \in
\overline{\Bb}(G,k)$ are contained in $\overline{A(T)}^B$ for a
suitable standardization $(T,B)$ (cf. Remark
\ref{rmk-standardization}), this allows us to define the map
$\overline{\Theta}$ by gluing together the maps
$\overline{\Theta}_{(T,B)}$.

\vskip 2mm

2) Equivariance. We now check that the map $\overline{\Theta}$ is
$G(k)\times G(k)$-equivariant, and for this we pick $(x,y) \in
\Bb(G,k) \times \overline{\Bb}(G,k)$ and choose a standardization
$(T,B)$ such that the partially compactified apartment
$\overline{A}^B$ for $A = A(T)$ contains both $x$ and $y$, as well
as a sequence $(y_n)_{n \geqslant 0}$ in $A$ converging to $y$. By
the Bruhat decomposition theorem for compactifed buildings, proved
in \cite[Prop. 4.20]{RTW1}, we can write $G(k) = G_x(k)NG_x(k)$,
where $G_x(k) = {\rm Stab}_{G(k)}(x)$ and $N = {\rm Norm}_G(T)(k)$.
Therefore, it is enough to prove that
$$\overline{\Theta}(gx,y) = \overline{\Theta}(x,y)g^{-1} \mbox{ and } \overline{\Theta}(x,hy) = h\overline{\Theta}(x,y)$$
for  $g$ and $h$ belonging to $G_x(k)$ or  $N$. If $g \in N$, then
$gx \in A$ and therefore \begin{eqnarray*} \overline{\Theta}(gx,y) &
= & \overline{\Theta}_{(T,B)}(gx,y) =  \lim_n
  \Theta(gx,y_n)  =  \lim_n \Theta(x,y_n)g^{-1}\\ &=&
  \overline{\Theta}_{(T,B)}(x,y)g^{-1}  =
  \overline{\Theta}(x,y)g^{-1}.\end{eqnarray*}
If $g \in G_x(k)$, then \begin{eqnarray*} \overline{\Theta}(gx,y) &
= & \overline{\Theta}(x,y)
   =  \overline{\Theta}_{(T,B)}(x,y) \\ &  = & \lim_n \Theta(x,y_n)
  = \lim_n \Theta(gx,y_n)  =  \lim_n \Theta(x,y_n)g^{-1} \\ & =
  & \overline{\Theta}_{(T,B)}(x,y)g^{-1} =  \overline{\Theta}(x,y)g^{-1}.
\end{eqnarray*}
If $h \in N$, then $h\overline{A}^B=\overline{A}^{hBh^{-1}}$ and
\begin{eqnarray*}
  \overline{\Theta}(x,hy) & = & \overline{\Theta}_{(T,hBh^{-1})}(x,hy)  =
  \lim_n \Theta(x,hy_n)  =  \lim_n h\Theta(x,y_n) \\ & = & h
  \overline{\Theta}_{(T,B)}(x,y)  =  h \overline{\Theta}(x,y).
\end{eqnarray*}
If $h \in G_x(k)$, then the points $x=hx$ and $hy_n$ are contained
in the apartment $hA$ and therefore
\begin{eqnarray*}
\overline{\Theta}(x,hy) & = & \overline{\Theta}_{(hTh^{-1},
hBh^{-1})}(x,hy)  = \lim_n \Theta(x,hy_n) =  \lim_n h \Theta(x,y_n)
\\ & = & h \overline{\Theta}_{(T,B)}(x,y) =  h
\overline{\Theta}(x,y).
\end{eqnarray*}

\vskip 2mm

3) Continuity. Let us now prove continuity of $\overline{\Theta}$.
The canonical map
$$\bigl(G(k) \times G(k)\bigr) \times (A \times \overline{A}^B) \rightarrow \Bb(G,k) \times \overline{\Bb}(G,k)$$
identifies the right-hand-side with a topological quotient of the
left-hand-side. By construction and equivariance, the map
$\overline{\Theta}$ is induced by the continuous map
$$\bigl(G(k) \times G(k)\bigr) \times (A \times \overline{A}^B) \rightarrow \Gbar, \ \ \ \bigl((g,h),(x,y) \bigr) \mapsto h\overline{\Theta}_{(T,B)}(x,y)g^{-1},$$
hence it  is continuous.

\vskip 2mm

4) Field extensions. Finally, the map $\overline{\Theta}$ is
compatible with non-archimedean field extensions since this is the
case for each map $\overline{\Theta}_{(T,B)}$ thanks to the formula
used to define it in the proof of Lemma \ref{lemma - explicit
formula seminorms}. \hfill $\Box$ \end{proof}

\section{Analytic strata in boundary divisors}
\label{s - boundary}
% [\emph{Geometric description of $X^{\rm an}$}]

In this section, we analyze the compatibility between the boundaries
at infinity of the Satake-Berkovich compactifications of Bruhat-Tits
buildings and of the wonderful compactifications. For this, we need
to recall some facts about the combinatorics and geometry of
boundaries of wonderful compactifications $\overline G$, which
amounts to decomposing the latter varieties into $G \times
G$-orbits. Our general reference is \cite[Section 3, p.
617]{BrionJAG}.

\vskip 2mm

Let ${\rm Par}(G)$ be the scheme of parabolic subgroups of $G$. The
\emph{type} $\tau = \tau(P)$ of a parabolic subgroup $P$ of $G$ is
the connected component of ${\rm Par}(G)$ containing $P$; we denote
by ${\rm Par}_\tau(G)$ this connected component. Since $G$ is split,
each connected component of ${\rm Par}(G)$ contains a $k$-rational
point. Let $\cT = \pi_0\bigl({\rm Par}(G)\bigr)$ denote the set of
types of parabolic subgroups. This set is partially ordered as
follows: given two types $\tau$ and $\tau'$, we set $\tau \leqslant
\tau'$ if there exist $P \in {\rm Par}_\tau(G)(k)$ and $P' \in {\rm
Par}_{\tau'}(G)(k)$ with $P \subset P'$. The minimal type
corresponds to Borel subgroups and the maximal type corresponds to
the trivial parabolic subgroup $G$. This set is also equipped with
an involution $\tau \mapsto \tau^{\rm opp}$ defined as follows: pick
a parabolic subgroup $P \in {\rm Par}_\tau(G)(k)$ as well as a Levi
subgroup $L$ of $P$ and set $\tau^{\rm opp} = \tau(P^{\rm opp})$,
where ${\rm P}^{\rm opp}$ is the only parabolic subgroup of $G$ such
that $P \cap {\rm P}^{\rm opp} = L$. Note that the type $\tau^{\rm
opp}$ is well-defined since $G(k)$ acts transitively by conjugation
on pairs $(P,L)$ consisting of a parabolic subgroup of type $\tau$
and a Levi subgroup $L$ of $P$.

\vskip 2mm

Let us go back now to the problem of decomposing $\overline G$ as
explicitly as possible into $G \times G$-orbits. We pick a
standardization $(T,B)$ of $G$ and use the associated notation as in
the previous section, such as the root system $\Phi$ and its
positive and negative subsets $\Phi^+$ and $\Phi^-$. We let also
$\Delta \subset \Phi_-$ denote the corresponding set of simple roots
and we recall that there is an increasing one-to-one correspondence
between the types of parabolics introduced above and the subsets of
$\Delta$: this map sends the type $\tau$ of a parabolic subgroup $P$
containing $B$ to $\Delta \cap \Phi(T,L)$, where $L$ is the Levi
subgroup of $P$ containing $T$; in particular, the type of Borel
subgroups (resp. of the trivial subgroup $G$) goes to $\varnothing$
(resp. to $\Delta$). The choice of $(T,B)$ gives us the ``partially
compactified big cell" $\overline G_0$, which can be identified with
$U_- \times Z \times U_+$ \cite[Lemmas 2.1 and 2.2]{Strickland} via
the natural open immersion
$$\varphi : U_- \times Z \times U_+ \rightarrow \overline G_0, \ \ (u_-,z,u_+) \mapsto u_- z u_+^{-1}.$$

\noindent One can choose specific $1$-parameter subgroups
$\lambda_\tau : {\bf G}_m \rightarrow T$ defined by
$\alpha(\lambda_\tau)(t) = 1$ if $\alpha \in \tau$ and
$\alpha(\lambda_\tau)(t) = t$ if $\alpha \in \Delta \setminus \tau$.

Each of these cocharacters has a limit at $0$, i.e. extends to a
morphism from ${\bf A}^1_k$ to $Z$. We set
$$e_{(T,B),\tau} = \lim_{t \to 0} \, \lambda_\tau(t) \in Z(k)$$ and we
note that this $k$-rational point can be described by
$$\alpha(e_{(T,B),\tau}) = 0 \ \mbox{ for all }  \alpha \in -\Phi(T,R_u(P))  \quad \mbox{ and } \quad \alpha(e_{(T,B),\tau}) = 1 \mbox{ for all } \alpha \in \Phi(T,L)^-.$$
The points $\{ e_{(T,B),\tau} \}_{\tau \subset \Delta}$ are
extremely useful because they provide a complete set of
representatives:
\begin{itemize}
\item for the $T \times T$-action on the toric affine variety $Z$,
\item for the $G \times G$-action on the wonderful compactification $\overline G$.
\end{itemize}
Therefore we obtain a one-to-one correspondence between these two
sets of orbits.

\vskip 2mm

Let $P$ denote the unique parabolic subgroup of $G$ of type $\tau$
containing $B$ (it can be described as consisting of the elements $g
\in G$ such that the limit $\lambda_\tau(t) g \lambda_\tau(t)^{-1}$
exists as $t \to 0$), and let $L$ be its Levi subgroup containing
$T=Z_G(T)$. Then $P^{\rm opp}$ is the parabolic subgroup in $G$
opposite $P$ with respect to $B$  (it can be described as consisting
of the elements $g \in G$ such that the limit $\lambda_\tau(t) g
\lambda_\tau(t)^{-1}$ exists as $t \to \infty$, and we have $P \cap
P^{\rm opp}=L$). We have the following description of stabilizers:
$${\rm Stab}_{G \times G}(e_{(T,B),\tau}) = {\rm diag}(L) \bigl(
R_u(P) Z(L) \times R_u(P^{\rm opp}) Z(L) \bigr) \subset P \times
P^{\rm opp}.$$ In other words, the wonderful compactification
$\overline G$ has a $G \times G$-equivariant stratification by
locally closed subspaces $\{X(\tau)\}_{\tau \in \cT}$ and each
stratum $X(\tau)$ is a homogeneous space under $G \times G$ which
comes with a $G \times G$-equivariant map
\[\pi_\tau : X(\tau) \rightarrow {\rm Par}_\tau(G) \times {\rm Par}_{\tau^{\rm opp}}(G)\]
sending the point $e_{(T,B),\tau}$ to $(P,P^{\rm opp})$ (note that
this map is well-defined since the stabilizer of $e_{(T,B),\tau}$ is
contained in $P \times P^{\rm opp}$). Moreover, for each point
$(P,P') \in {\rm Par}_\tau(G)(k) \times {\rm Par}_{\tau^{\rm
opp}}(G)(k)$ consisting of two opposite parabolic subgroups with
respect to a common Levi subgroup $L$, the fiber of $\pi_\tau$ over
$(P,P')$ is canonically isomorphic to the adjoint quotient $L/Z(L)$
of $L$.

\vskip 2mm

One can also give an explicit description of the intersection of
$X(\tau)$ with $\overline G_0$ and of the restriction of $\pi_\tau$
to $X(\tau) \cap \overline G_0$. For simplicity, let us write
$\Phi(Q) = \Phi(T,Q)$ for every subgroup $Q$ of $G$ containing the
torus $T$. The stratum $X(\tau)$ intersects the toric variety $Z =
{\rm Spec}k[\langle \Phi^-\rangle]$ along the locally closed
subspace
\[Z(\tau) = \{z \in Z: \alpha(z) = 0 \mbox { for all } \alpha \in -\Phi(R_u(P)) \mbox{ and } \alpha(z) \neq 0 \mbox{ for all } \alpha \in \Phi(L)^-\},\]
i.e. the intersection of the vanishing sets of all negative roots
belonging to the unipotent radical of $P^{\rm opp}$ and the
non-vanishing set of all negative roots belonging to the Levi
subgroup $L$. This stratum $Z(\tau)$ is a principal homogeneous
space under $T/T(\tau)$, where $T(\tau)$ is the subtorus given as
the connected component of the kernel of all $\alpha \in \Phi(L)$,
and $Z(\tau)$  is trivialized by the $k$-rational point
$e_{(T,B),\tau}$. The torus $T(\tau)$ is the center of $L$, hence
$T/T(\tau)$ is the maximal split torus of $L/Z(L)$ induced by $T$.

\vskip 2mm

At last, the following diagram
$$\xymatrix{
\mbox{$\prod\limits_{\alpha \in - \Phi(R_u(P))} \!\!\!U_\alpha
\times \left( \prod\limits_{\alpha \in \Phi(L)^-} \!\!\!U_\alpha
\times Z(\tau) \times \prod\limits_{\alpha \in \Phi(L)^+}
\!\!\!U_\alpha \right) \times \prod\limits_{\alpha \in \Phi(R_u(P))}
\!\!\!U_\alpha$} \ar@{->}[d]_{({\rm pr}_1,{\rm pr}_3)}
\ar@{->}[r]^{\quad \quad \quad \quad \quad \quad \quad \quad \quad
\quad \quad \quad \varphi}& X(\tau) \cap \overline G_0
 \ar@{->}[d]^{\pi_\tau} \\
\prod\limits_{\alpha \in - \Phi(R_u(P))} \!\!\!U_\alpha \times
\prod\limits_{\alpha \in \Phi(R_u(P))} \!\!\!U_\alpha \ar@{->}[r] &
{\rm Par}_\tau(G) \times {\rm Par}_{\tau^{\rm opp}}(G) }$$ where the
bottom horizontal map is $(u_-,u_+) \mapsto (u_- P u_-^{-1}, u_+
P^{\rm opp} u_+^{-1})$, is commutative.

Let $P$ be a parabolic subgroup of $G$ of type $\tau$, and let
$\lambda: G \rightarrow G/P$ be the projection to the associated
flag variety, which is isomorphic to $\mathrm{Par}_\tau(G)$. Recall
the embedding $\vartheta: \Bb(G,k) \rightarrow G^\an$ defined in
\cite[Prop. 2.4]{RTW1}. By composition, we get a map $\vartheta_\tau
= \lambda^\an \circ \vartheta: \Bb(G,k) \rightarrow \mathrm{Par}_\tau(G)^\an$, which is $G(k)$-equivariant
and  independent of the choice of the parabolic $P$ of type $\tau$
by \cite[Lemma 2.13]{RTW1}. If $\tau$ is the type of a Borel
subgroup, we have seen this map under the name $\vartheta_\emptyset$
already in section \ref{s - embedding}.

\begin{Prop}
\label{prop - geometric.description} Let $P$ be any parabolic
subgroup of $G$ of type $\tau (\neq \Delta)$, giving rise to the
boundary stratum $\Bb(P_{ss},k)$ lying in $\bbart \setminus
\Bb(G,k)$. \vskip1mm
\begin{itemize}
\item[{\rm (i)}]~The map $\overline{\Theta}$ sends $\mathcal{B}(G,k) \times \mathcal{B}(P_{ss},k)$ into $X(\tau)^{\rm an}$.
\item[{\rm (ii)}]~We have
$$(\pi_\tau^\an \circ \overline{\Theta})(x,y) = (P,\vartheta_{\tau^{\rm opp}}(x))$$
for all $(x,y) \in \mathcal{B}(G,k) \times \mathcal{B}(P_{ss},k)$.
\item[{\rm (iii)}]~ For every point $x \in \mathcal{B}(G,k)$, the restriction of $\overline{\Theta}(x, \cdot)$ to $\mathcal{B}(P_{ss},k)$ is a continuous embedding.
%\item[(iv)] For any points $x,x' \in \mathcal{B}(G,k)$ such that $\theta_{\tau^{\rm opp}}(x) = \theta_{\tau^{\rm opp}}(x')$, the maps $\overline{\Theta}(x,\cdot)$ and $\overline{\Theta}(x',\cdot)$ have the same restriction to the stratum $\mathcal{B}(P_{ss},k)$.
\end{itemize}
\end{Prop}

\vskip 3mm
\begin{Remark} {\rm Note that in assertion (ii) above, $P$ is a $k$-rational point in ${\rm Par}_\tau(G)^\an $ whereas the point $\vartheta_{\tau^{\rm opp}}(x)$ in $ {\rm Par}_{\tau^{\rm opp}}(G)^\an$ is defined over a transcendental non-archimedean field extension. One should also be aware that, if one denotes by $|X|$ the underlying topological space of a non-Archimedean analytic space $X$, then $|X \times Y|$ is in general different from $|X| \times |Y|$. However, since $X(k) \times |Y| \subset |X \times Y|$, the formula in point (ii) does make
sense.}
\end{Remark}

\begin{proof}[Proof of Proposition \ref{prop -
geometric.description}.]
We fix a standardization $(T,B)$ of $G$ and
use the notation introduced above.

\vskip 2mm

Let us prove (i) and (ii). The partially compactified apartment
$\overline{A}^B$ intersects $\mathcal{B}(P_{ss},k)$ along the
subspace $\overline{A}^B(P)$ defined by the conditions $\alpha = 0$
for each root $\alpha$ in $-\Phi(R_u(P))$ and $\alpha >0$ for each
root $\alpha \in \Phi(L)^-$. This is the apartment of the maximal
split torus $T/T(\tau)$ of $P_{ss}$. According to the explicit
formula for $\overline{\Theta}_{(T,B)}$ in Proposition \ref{prop -
extension.apartment}, a point $(x,y) \in A \times \overline{A}^B(P)$
is mapped to the Gauss point  of
$$\left( \prod_{\alpha \in -\Phi(R_u(P))} U_\alpha \times \Bigl(\prod_{\alpha \in \Phi(L)^-} U_\alpha \times Z \times \prod_{\alpha \in \Phi(L)^+} U_\alpha\Bigr) \times \prod_{\alpha \in \Phi(R_u(P))} U_\alpha \right)^{\rm an}$$
defined by
$$|\alpha| = \langle y, \alpha \rangle \langle x,\alpha\rangle^{-1} \mbox{ for all } \ \alpha \in \Phi^-,$$
which vanishes if and only if $\alpha \in -\Phi(R_u(P))$, and by
$$|\xi_\alpha|
= \left\{\begin{array}{ll} 0 & \mbox{ if } \alpha \in -\Phi(R_u(P)) \\ \langle y ,\alpha \rangle & \mbox{ if } \ \alpha \in \Phi(L)^{-} \\
\langle x,\alpha\rangle & \mbox{ if } \ \alpha \in \Phi(L)^+ \cup
\Phi(R_u(P)). \end{array} \right.$$ By the commutative diagram
preceding our statement, we thus have
$\overline{\Theta}_{(B,T)}(x,y) \in (X(\tau) \cap \overline
G_0)^\an$. Using the explicit formula for $\vartheta_{\tau^{\rm
opp}}(x)$ from the proof of \cite[Lemma 3.33]{RTW1}, we also find
that $\pi_\tau^\an(\overline{\Theta}_{(T,B)}(x,y))$ is the point
$(\{P\},\vartheta_{\tau^{\rm opp}}(x))$ of ${\rm Par}_\tau(G)^{\rm
an} \times {\rm Par}_{\tau^{\rm opp}}(G)^{\rm an}$. We have thus
proved (i) and (ii) for the restriction of $\overline{\Theta}$ to $A
\times \overline{A}^B(P)$. Since $\pi_\tau$ is $G \times
G$-equivariant, the general case follows via translation by the
subgroup $G \times P$.

\vskip 2mm

Let us finally prove (iii). Consider two points $y,y'$ in
$\mathcal{B}(P_{ss},k)$ such that $\overline{\Theta}(x,y) =
\overline{\Theta}(x,y')$. Pick a maximal split torus $T$ of $G$
contained in $P$ such that $y$ and $y'$ belong to the closure of
$A=A(T)$; for any choice of a Borel subgroup $B$ of $P$ containing
$T$, we have $y,y' \in \overline{A}^B$. Choose also some $g \in
G(k)$ such that $x$ belongs to $g^{-1}A$. By assumption, we have
$$\overline{\Theta}(gx,y) = \overline{\Theta}(x,y)g^{-1} = \overline{\Theta}(x,y')g^{-1} = \overline{\Theta}(gx,y').$$
Since $gx, y$ and $y'$ are all contained in $\overline{A}^B$, we
deduce
$$\overline{\Theta}_{(T,B)}(gx,y) = \overline{\Theta}_{(T,B)}(gx,y')$$
and therefore $y=y'$ by injectivity of
$\overline{\Theta}_{(T,B)}(gx,-)$ on $\overline{A}^B$.
%\vskip1mm We argue in a similar way to prove that the image of
%$\overline{\Theta}(x_0,\cdot)$ is closed. We first check via the
%explicit formula that
%$\overline{\Theta}_{(T,B)}(x_0,\overline{g^{-1}A}^g^{-1}Bg)$ is
%closed and then deduce that the image of $\Theta(x_0,\cdot)$, which
%is a locally finite union of translates
%of$\overline{\Theta}_{(T,B)}(gx_0,\overline{g^{-1}A}^g^{-1}Bg)$, is
%closed.
\hfill $\Box$ \end{proof}

\section{Equivariant comparison}
\label{s - equivariant comparison}

Thanks to the better understanding of the relationship between the
boundaries provided by the previous section, we can now prove our
main comparison theorem, stated as the first two points of the main
theorem in the introduction.
\begin{Thm}
\label{thm:comparison}
\begin{itemize}
\item[{\rm (i)}]~For every point $x \in \mathcal{B}(G,k)$, the map
\[\overline{\Theta}(x,\cdot):\bbart \rightarrow \Gbar\]
is a continuous, $G(k)$-equivariant embedding.
\item[{\rm (ii)}]~Assume that $k$ is locally compact.
Then $\overline{\Theta}(x,\cdot)$ is a closed embedding, and the
compactified building $\bbart$ is homeomorphic to the closure of the
image of the building $\Bb(G,k)$ under the embedding
\[\Bb(G,k) \stackrel{\Theta(x,\cdot)}{\longrightarrow}  G^\an \hookrightarrow \Gbar.\]
\end{itemize}
\end{Thm}

\begin{proof} Injectivity of $\overline{\Theta}(x,\cdot)$ follows immediately
from the following three observations based on Proposition \ref{prop
- geometric.description}. For any parabolic subgroup $P$ of $G$ of
type $\tau$ :
\begin{itemize}
\item the map $\overline{\Theta}(x,\cdot)$ sends $\mathcal{B}(P_{ss},k)$ into $X(\tau)^{\rm an}$;
\item the image of $\overline{\Theta}(x,\mathcal{B}(P_{ss},k))$ under the map $\pi_\tau$ is contained into $\{P\} \times {\rm Par}_{\tau(P)^{\rm opp}}(G)^{\rm an}$;
\item the map $\overline{\Theta}(x,\cdot)$ restricts injectively to the stratum $\mathcal{B}(P_{ss},k)$ of $\overline{\mathcal{B}}(G,k)$ associated with $P$.
\end{itemize}
Note that, thanks to the first two observations, different strata
have disjoint images.

\vskip 2mm If $k$ is locally compact, the map
$\overline{\Theta}(x,\cdot)$ is closed since it is continuous,
$\overline{\mathcal{B}}(G,k)$ is compact by \cite[Prop. 3.34]{RTW1},
and $\Gbar$ is Hausdorff. Since $\Bb(G,k)$ is dense in $\bbart$ by
\cite[Prop. 3.34]{RTW1}, the last claim follows. \hfill $\Box$
\end{proof}

In order to complete the proof of our main theorem stated in the
introduction, it remains to show part (iii). Strictly speaking, this
statement deals with orbits closures while Proposition \ref{prop -
geometric.description} deals with the orbits themselves. The
relationship is in fact very neat since it is well-known that, with
our notation, we have by \cite[Th. 3.9]{ConciniSpringer}:
$$\overline{X(\tau)} = \bigcap_{i \in \Delta \setminus \tau} D_i,$$
hence $$\overline{X(\tau)} = \bigcup_{\tau' \leqslant \tau}
X(\tau').$$
%Moreover the fibration  \[\pi_\tau : X(\tau) \rightarrow {\rm Par}_\tau(G) \times {\rm Par}_{\tau^{\rm opp}}(G)\] of the previous section, whose fibers are adjoint semisimple quotients of Levi factors, extends to a fibration  \[\overline\pi_\tau : \overline{X(\tau)} \rightarrow {\rm Par}_\tau(G) \times {\rm Par}_{\tau^{\rm opp}}(G).\] with fibers isomorphic to wonderful compactifications of the latter groups, see e.g. \cite[Th. 2.26]{EvensJones}.

%\vskip 3mm

%If we take into account Bruhat-Tits buildings, we can start with the decompositions: $$\overline G = \bigsqcup_{\tau \in \mathcal{T}} X(\tau) = \bigcup_{\tau \in \mathcal{T}} \overline{X(\tau)},$$ where the adjoint semisimple quotient $P_{ss}$ is seen as a boundary symmetric space and $\overline{P_{ss}}$ is its wonderful compactification, and take their analytifications.
By Prop. \ref{prop - geometric.description} (ii), this observation
implies that, for every parabolic subgroup $P$ of type $\tau$,
$$\Bb(P_{ss},k) \subset X(\tau)^{\rm an} \quad {\rm and} \quad \overline\Bb(P_{ss},k) \subset \overline{X(\tau)}^{\rm an}$$
since $$\overline{\mathcal{B}}(P_{ss},k) = \bigcup_{P' \in {\rm
Par}(G)(k), \ P' \subset P} \mathcal{B}(P'_{ss},k)$$ by
\cite[Theorem 4.1 and Example 3.9.(ii)]{RTW1}.

\vskip 5mm We conclude by a strengthening of assertion (i) in our
Theorem, thereby answering a question raised by one of the referees.

\begin{Prop} The map $$\overline{\Theta} : \mathcal{B}(G,k) \times \overline{\mathcal{B}}(G,k) \rightarrow \overline{G}^{\rm an}$$
is injective.
\end{Prop}

\begin{proof} (Step 1) Let us first check that the map $\Theta :
\mathcal{B}(G,k) \times \mathcal{B}(G,k) \rightarrow G^{\rm an}$ is
injective, a fact which we did not notice in \cite{RTW1}. Consider
two points $(x,y), (x',y') \in \mathcal{B}(G,k) \times
\mathcal{B}(G,k)$ such that $\Theta(x,y) = \Theta(x',y')$. By
\cite[Lemma 2.10]{RTW1}, this condition still holds after an
arbitrary non-Archimedean extension of $k$ and therefore we may
assume that there exists $g', h,h' \in G(k)$ such that $x'=g'x,
y=hx$ and $y'=h'x$. By \cite[Proposition 2.11]{RTW1}, we get
$$h\Theta(x,x) = h'\Theta(x,x){g'}^{-1},$$ hence $$h^{-1}h'
\vartheta(x) = \vartheta(x)g' \ \ \ \ {\rm and} \ \ \ h^{-1}h'G_x =
G_x g'$$ by \cite[Definition 2.9 and Proposition 2.4.(i)]{RTW1}. We
thus can write $g' = h^{-1}h's$ for some suitable $s \in G_x(k)$ and
deduce $$g'G_x = h^{-1}h'G_x = G_x g', \ \ \ {\rm hence} \ \
G_{g'\cdot x} = G_x.$$ This implies $g'\cdot x =x$ by
\cite[Corollary 2.5]{RTW1}, and therefore $g' \in G_x(k)$. It
follows that $h^{-1}h'$ also belongs to $G_x(k)$, hence $x' = g'x =
x$ and $y' = h'x = h(h^{-1}h')x = hx = y$.

\vskip2mm (Step 2) Let us now prove that $\overline{\Theta}$ is also
injective. It follows easily from Proposition 3.1 (i) and (ii) that
the map $\overline{\Theta}$ separates the strata $\mathcal{B}(G,k)
\times \mathcal{B}(P_{ss},k)$ associated with the parabolic
subgroups $P$ of $G$, so it is enough to check that its restriction
to each stratum $\mathcal{B}(G,k) \times \mathcal{B}(P_{ss},k)$ is
injective. We remark that it is easy if the parabolic subgroup $P$
does not contain any simple factor of $G$, i.e. if its type $\tau =
\tau(P)$ is nondegenerate. In this case indeed, the opposite type
$\tau^{\rm opp}$ is also nondegenerate, hence the map
$\vartheta_{\tau^{\rm opp}}$ from $\mathcal{B}(G,k)$ to ${\rm
Par}_{\tau^{\rm opp}}(G)^{\rm an}$ is injective \cite[Proposition
3.29]{RTW1} and thus the conclusion follows from Proposition
3.1.(ii) and (iii).

\vskip2mm (Step 3) In general, one can write $G = G_1 \times G_2$
for some semisimple groups of adjoint type $G_1$ and $G_2$ such that
$P = G_1 \times P_2$, where $P_2$ is a nondegenerate parabolic
subgroup of $G_2$ of type $\tau_2$. The schemes ${\rm Par}_\tau(G)
\times {\rm Par}_{\tau^{\rm opp}}(G)$ and ${\rm Par}_{\tau_2}(G_2)
\times {\rm Par}_{\tau_2^{\rm opp}}(G_2)$ are canonically
isomorphic, the building of $G$ is the product of the buildings of
$G_1$ and $G_2$, and $\vartheta_\tau = \vartheta_{\tau_2} \circ
pr_2$. Fix a point $x_2$ in $\mathcal{B}(G_2,k)$. By Proposition
3.1(ii), the restriction of $\overline{\Theta}$ to
$$\mathcal{B}(G_1,k) \times \{x_2\} \times \mathcal{B}(P_{ss},k) =
\mathcal{B}(G_1,k) \times \{x_2\} \times \mathcal{B}(G_1,k) \times
\mathcal{B}(P_{2,ss},k)$$ is a $G_1 \times P_{ss}$-equivariant map
whose image is contained in the fiber of $\pi_\tau$ over the point
$(P,\vartheta_{\tau_2}(x_2))$, which is a canonically split torsor
under $G_1^{\rm an} \times (P_{2,ss}/Z(P_{2,ss}))^{\rm an}$.
Projecting onto ${\rm G}_1^{\rm an}$, we thus obtain a $G_1 \times
G_1$-equivariant map $$M : \mathcal{B}(G_1,k) \times
\mathcal{B}(G_1,k)  \times \mathcal{B}(P_{2,ss},k) \rightarrow (G_1
\otimes_k K)^{\rm an},$$ where $K$ is equal to the completed residue
field $\mathcal{H}(\vartheta_{\tau_2}(x_2))$. Choosing a
standardization $(B,T) = (B_1 \times B_2,T_1 \times T_2)$ of $G$ and
computing in the associated big cell as in the proof of Proposition
3.1, we observe that the restriction of $M$ to $$A \times
\overline{A}^{B}(P) = A(T_1) \times A(T_2) \times A(T_1) \times
\overline{A}^{B_2}(P_2)$$ factors through the canonical projection
onto $A(T_1) \times A(T_1)$, and that the induced map coincides with
the restriction of $\Theta_{G_1 \otimes_k K}$ to $A(T_1) \times
A(T_1)$. By translation, we deduce that $M$ factors through the
projection onto $\mathcal{B}(G_1,k) \times \mathcal{B}(G_1,k)$ and
that the induced map to $G_{1,K}^{\rm an}$ coincides with
$\Theta_{G_1 \otimes_k K}$.

\vskip2mm (Step 4) We can now finish the proof. We make the
identification $$\mathcal{B}(G,k) \times \mathcal{B}(P_{ss},k) =
\mathcal{B}(G_1,k) \times \mathcal{B}(G_2,k) \times
\mathcal{B}(G_1,k) \times \mathcal{B}(P_{2,ss},k)$$ and consider two
points $(x_1,x_2,y_1,y_2), (x_1',x_2',y_1',y_2')$ in the product
such that $\overline{\Theta}((x_1,x_2),(y_1,y_2))=
\overline{\Theta}((x_1',x_2'),(y_1',y_2'))$. Projecting this
equality by $\pi_\tau$, we deduce $\vartheta_{\tau_2}(x_2) =
\vartheta_{\tau_2}(x_2')$, hence $x_2'=x_2$. Combining steps 1 and
3, we then get $(x_1',y_1') = (x_1,y_1)$. We can then apply
Proposition 3.1.(iii) to derive $(y_1',y_2') = (y_1,y_2)$, hence
$(x_1',x_2',y_1',y_2') = (x_1,x_2,y_1,y_2)$. \hfill $\Box$
\end{proof}

\noindent {\bf Further questions:} Our constructions are all
Galois-equivariant and can be descended to ground fields over which
the group $G$ need not be split. On the one hand, it seems to us
that wonderful compactifications of non-split groups are less
explicitly described in the literature, presumably due to lack of
representation-theoretic motivation. On the other hand, descent in
Bruhat-Tits theory is a central topic. The geometric description of
Satake-Berkovich compactifications of Bruhat-Tits buildings could be
a useful tool to describe wonderful compactifications of non-split
groups.

Another interesting line of further research is the generalization
of our results to other equivariant compactifications of the
reductive group $G$ \cite{Timashev}.

\vskip 5mm

\noindent {\bf Acknowledgements.} We thank the anonymous referees
for their questions and  remarks, which helped us to improve the
paper in many ways. The first author is very grateful to the
Humboldt Foundation for its support, he would like to thank the
participants of the Institut Fourier workshop ``Symmetric spaces"
held a few years ago and the Institut f\"ur Mathematik (FB12) of the
Goethe-Universit\"at for his optimal stay in Frankfurt during the
Fall Semester of 2017. The second author was partially supported by
the GeoLie project of the Agence Nationale de la Recherche (project
ANR-15-CE40-0012). The third author would like to thank Deutsche
Forschungsgemeinschaft for supporting this joint work under grant WE
4279/7.

\providecommand{\bysame}{\leavevmode\hbox to3em{\hrulefill}\thinspace}
%\providecommand{\MR}{\relax\ifhmode\unskip\space\fi MR }
% \MRhref is called by the amsart/book/proc definition of \MR.
%\providecommand{\MRhref}[2]{%
%  \href{http://www.ams.org/mathscinet-getitem?mr=#1}{#2}
%}
%\providecommand{\href}[2]{#2}
%\begin{thebibliography}{{Kem}92}
%
%
%\end{thebibliography}

\bibliographystyle{amsalpha}
\bibliographymark{References}
% This would change the heading "References" by "Bibliography"
% \renewcommand{\refname}{Bibliography}

\end{document}